\documentclass[10pt]{article}
\usepackage{graphicx}
\usepackage{amsmath, amsthm, amssymb,amsfonts}
\usepackage[mathscr]{euscript}
\usepackage{cleveref}
\usepackage{xcolor}
\usepackage{enumitem,bm}
\usepackage{mathtools,accents,upgreek,bm}
\usepackage{mathrsfs,comment,aliascnt,braket,esint}
\usepackage{microtype}

\usepackage{geometry}
\newcommand{\cC}{\mathcal{C}}
\DeclareTextFontCommand{\dc}{\color{purple}}

\newcommand{\K}{\mathscr{K}}
\newcommand{\m}{\mathfrak{m}}
\newcommand{\M}{{\mathit{M}}}
\DeclareMathOperator{\diam}{diam}
\renewcommand{\d}{\,\mathrm{d}}

\newcommand{\Test}{\mathrm{Test}}
\newcommand{\e}{\mathrm{e}}

\newcommand{\D}{\mathscr{D}}

\newcommand{\loc}{{\mathrm{loc}}}

\DeclareMathOperator{\RCD}{RCD}

\renewcommand{\k}{\mathrm{k}}
\renewcommand{\P}{\mathsf{P}}
\newcommand{\One}{1}

\newcommand{\dist}{\mathsf{d}}
\newcommand{\df}{\coloneqq}
\newcommand{\Ch}{{\sf Ch}}
\DeclareMathOperator{\Lip}{Lip}
\DeclareMathOperator{\AC}{AC}

\DeclareMathOperator{\RR}{\mathbb{R}}

\newcommand{\BE}{{\mathrm{BE}}}

\theoremstyle{plain}
\newtheorem{theorem}{Theorem}[section]
\newtheorem{proposition}[theorem]{Proposition}

\newtheorem{lemma}[theorem]{Lemma}

\theoremstyle{definition}
\newtheorem{definition}[theorem]{Definition}
\newtheorem{remark}[theorem]{Remark}

\newcommand{\contpath}{\cC([0,1],M,\dist)}
\newcommand{\contpathxy}{\cC_{xy}([0,1],M,\dist)}
\newcommand{\accurves}{\AC^2(M,\dist)}
\newcommand{\accurvesxy}{\AC^2_{xy}(M,\dist)}

\newcommand{\Ltwo}{L^2(M,\m)}
\newcommand{\Linf}{L^{\infty}(M,\m)}
\newcommand{\Sob}{W^{1,2}(M,\dist,\m)}

\newcommand{\Lonew}{L^1(M,\m_w)}

\newcommand{\Linfw}{L^{\infty}(M,\m_w)}

\newcommand{\Ltwoloc}{L^2_{\text{loc}}(M,\m)}

\newcommand{\Sobloc}{W^{1,2}_{\text{loc}}(M,\dist,\m)}

\title{Doubling, Poincaré and Gromov--Hausdorff precompactness\\for tamed spaces}

\author{Mathias Braun\thanks{M.~Braun: Institute of Mathematics, EPFL, 1015 Lausanne, Switzerland. mathias.braun@epfl.ch} \and  Chiara Rigoni\thanks{C.~Rigoni: Faculty of Mathematics, University of Vienna, Oskar-Morgenstern-Platz 1, 1090 Vienna, Austria. chiara.rigoni@univie.ac.at} \and Christian Rose\thanks{C.~Rose: Insititut f\"ur Mathematik, Universit\"at Potsdam, Karl-Liebknecht-Stra\ss{}e 24-25, 14476 Potsdam, Germany. christian.rose@uni-potsdam.de} \and
  David Tewodrose\thanks{D.~Tewodrose: Department of Mathematics and Data Science, Vrije Universiteit Brussel, Pleinlaan 2, 1050 Elsene, Belgium. david.tewodrose@vub.be}
}

\begin{document}

\maketitle

\begin{abstract}
    We establish local volume doubling and a local $L^2$-Poincaré inequality for metric measure spaces tamed by a measure in the extended Kato class that satisfy the bounded interpolation property. To this end, we construct a bi-Lipschitz time change to a metric measure space satisfying a uniform lower Ricci curvature bound. As an application, we obtain pointed measured Gromov--Hausdorff precompactness for such spaces.\\


\end{abstract}

\tableofcontents

\section{Introduction}

Volume doubling and Poincaré inequalities play a central role in the analysis of metric measure spaces. These properties enable the extension of classical Euclidean analysis to this more general setting. As observed by Gromov \cite{Gromov}, volume doubling also ensures pointed measured Gromov--Hausdorff (pmGH) compactness and guarantee the existence of tangent cones at any point of a metric measure space. The study of tangent cones leads to structure results like rectifiability or stratification of the singular set for classes of singular spaces. In particular, this applies to sequences of metric measure spaces satisfying uniform constant lower curvature bounds. This was pioneered for Alexandrov spaces by Burago--Gromov--Perelman \cite{BuragoGromovPerelman} and Ricci limit spaces by Cheeger--Colding \cite{ChCo97}.

Riemannian manifolds with possibly unbounded yet $L^p$-small negative part of the Ricci curvature also satisfy volume doubling and Poincar\'e, as shown by Petersen--Wei \cite{PW1} and Dai--Wei--Zhang \cite{DWZ18}. More generally, extensions to so-called Kato-bounded Ricci curvature have been found in Rose \cite{Rose19}, Carron \cite{Carron-19}, and Carron--Mondello--Tewodrose \cite{CMT_KatoBE}. The Cheeger--Colding structure theory for corresponding pmGH limit spaces was developed by Carron--Mondello--Tewodrose in the articles \cite{CMT,CMT2,CMT_KatoBE}.

However, the space of manifolds is not closed in the pmGH topology. The natural class of spaces for pmGH convergence are metric measure spaces, whose study has received considerable attention during the last decades. For finite-dimensional RCD spaces, pmGH compactness and a corresponding Cheeger--Colding structure theory is nowadays well-established, cf.~Sturm \cite{sturm2006-ii}, Gigli \cite{gigli-splitting}, Mondino--Naber \cite{MondinoNaber}, and Brué--Semola \cite{BrueSemola}. RCD spaces are metric measure spaces with a generalized lower Ricci curvature (and upper dimension) bound in terms of optimal transport introduced by Sturm \cite{sturm2006-i,sturm2006-ii} and Lott--Villani \cite{lott-villani2009}, enhanced by the hypothesis of infinitesimal Hilbertianity due to Gigli \cite{Gigli15}. 

This motivates a Cheeger--Colding structure theory for spaces covering these scenarios: \emph{non\-smooth} spaces with \emph{nonconstant}, possibly unbounded lower Ricci curvature bounds. This article initiates this program by establishing volume doubling, Poincaré inequalities, and a fundamental pmGH compactness theorem. Our framework are metric measure spaces which are tamed in the sense of Erbar--Rigoni--Sturm--Tamanini \cite{ERST}, i.e.~they admit synthetic lower Ricci curvature bounds in the (extended) Kato class. Compared to RCD spaces, tamed metric measure spaces are expected to admit more diverse structures. For instance, Carron--Mondello--Tewodrose \cite{CMT_branching} show that non-collapsed limits of manifolds with a controlled Kato bound on the Ricci curvature may exhibit branching --- a phenomenon that does not occur for finite-dimensional RCD spaces, as shown by Deng \cite{deng}. Moreover, there are finite-dimensional RCD spaces which cannot arise as non-collapsed pmGH limits of Riemannian manifolds with a uniform Ricci curvature lower bound.

We now describe our framework more precisely. We consider a locally compact length metric measure space $(\M,\dist,\m)$ that is infinitesimally Hilbertian in the sense of Gigli \cite{Gigli15}; as usual, metric measure spaces are supposed here complete and separable. Their Cheeger energy $\Ch$ is a strongly local Dirichlet form; we assume it admits a heat kernel and the induced intrinsic distance (in the sense of Dirichlet forms) coincides with $\dist$. We also suppose $\m(M) < \infty$, although we believe this assumption can be dispensed. In addition, we assume $(\M,\dist,\m)$ satisfies the so-called \emph{bounded interpolation property} of Gigli--Nobili \cite{GigliNobili}. In a nutshell, this is a first-order condition that stipulates the rich existence of Wasserstein geodesics whose $\m$-densities are uniformly bounded; cf.~\Cref{Def:BIP} for details. Our last hypothesis is that $(M,\dist,\m)$ satisfies the  \emph{taming condition} $\BE(-\kappa,n)$ in the sense of Erbar--Rigoni--Sturm--Tamanini \cite{ERST}, where $\kappa$ is a Borel measure in the extended Kato class of $(\M,\dist,\m)$ and $n\in [1,\infty)$. As detailed in \Cref{Def:BE}, $\BE(-\kappa,n)$ means that for sufficiently many functions $f$, the Bochner inequality 
\begin{align*}
    \Big[\Delta\frac{\vert \nabla f\vert^2}{2} - \langle\nabla f,\nabla\Delta f\rangle\Big]\,\m\, \geq \frac{(\Delta f)^2}{n}\,\m - \vert\nabla f\vert^2\,\kappa
\end{align*}
holds in a weak sense; $\kappa$ belonging to the extended Kato class means that there exists $T>0$ such that the \emph{Kato constant} $\k_T(\M,\dist,\m,\kappa)$ from Definition \ref{def:Kato} satisfies $\k_T(\M,\dist,\m,\kappa)<1$.

We now introduce our main results. Given $\rho,\nu>0$, we say $(\M,\dist,\m)$ satisfies
\begin{itemize}
    \item a \emph{doubling property at scale $\rho$ and dimension $\nu$} if there exists a constant $c>0$ such that, for all $0<r\leq R\leq \rho $ and $x\in \M$, we have
    \[
    \frac{V(x,R)}{V(x,r)}\leq c\left(\frac{R}{r}\right)^{\nu}
    \]
    where $V(x,r)$ denotes the volume of the open ball $B(x,r)$ with center $x$ and radius $r$,
    \item an \emph{$L^2$-Poincaré inequality at scale $\rho$} if there are constants $c_1,c_2>0$ such that, for all $0<r\leq \rho$, all $x\in \M$, and all $\varphi\in W^{1,2}(\M,\dist,\m)$, we have 
    \[
    \inf_{a\in\RR}\int_{B(x,r)}|\varphi-a|^2\d\m\leq c_1r^2\int_{B(x,c_2 r)}|\nabla\varphi|^2\d\m.
    \]
    \end{itemize}

\begin{theorem}[Doubling and Poincaré for tamed metric measure spaces]\label{thm:MainA} 
Let $(\M,\dist,\m)$ be a locally compact length metric measure space with $\m(M)<\infty$ which is infinitesimally Hilbertian. Assume the Cheeger energy $\Ch$ admits a heat kernel and the induced intrinsic distance coincides with $\dist$. Moreover, let $(\M,\dist,\m)$ satisfy the bounded interpolation property. Lastly, assume there exist a Borel measure $\kappa$ and $n\in [1,\infty)$ such that $(\M,\dist,\m)$ satisfies $\BE(-\kappa,n)$ and there exist $T>0$ and $\delta \in [0, 1\wedge \frac{1}{2|n-2|})$ such that
\begin{equation*}
\k_T:=\k_T(\M,\dist,\m,\kappa) \le \delta \, \cdot 
\end{equation*}
Then there exist constants $c=c(\delta,n)\geq 1$ and $N = N(\delta,n) \geq n$ such that $(M,\dist,\m)$ satisfies a doubling property and an $L^2$-Poincaré inequality at scale $(T/(2\delta N))\wedge (\diam(\M)/c)$ and dimension $N$.
\end{theorem}

We derive the previous theorem by means of a suitable time change, as was done by Carron--Mondello--Tewodrose in the context of complete Riemannian manifolds \cite{CMT_KatoBE}; see also Bormann--Tewodrose \cite{BormannTewodrose} for the case of manifolds with boundary. The bounded interpolation property is used in the proof of Lemma \ref{lem:LYtimechange} where we derive a Harnack inequality for the time-changed space.

This approach yields quantified constants in the doubling and Poincaré properties, leading to the following precompactness statement.

\begin{theorem}[pmGH precompactness for tamed metric measure spaces]\label{th:MainB}
Consider $n \in [1,+\infty)$, $T \in (0,+\infty)$ and  $\delta \in [0, 1\wedge \frac{1}{2|n-2|})$.

\begin{enumerate}
    \item Let $\mathcal{M}(n,T,\delta)$ be the collection of isometry classes of pointed metric measures spaces $(M,\dist,\m,o)$ satisfying the assumptions of Theorem \ref{thm:MainA}. Then $\mathcal{M}(n,T,\delta)$ is precompact in the pointed measured Gromov--Hausdorff topology.
    \item For $\eta >0$, let $\mathcal{M}(n,T,\delta,\eta)$ be the collection of isometry classes of spaces $(M,\dist,\m,o,\Ch) \in \mathcal{M}(n,T,\delta)$ satisfying $\eta^{-1} \le \m(B_{\sqrt{T}}(o)) \le \eta$. Then $\mathcal{M}(n,T,\delta,\eta)$ is precompact in the pointed Mosco--Gromov--Hausdorff topology.
\end{enumerate}
\end{theorem}

The first part of the statement is a direct consequence of Gromov's compactness theorem. The second one follows from a variant of this theorem including Mosco convergence of the Cheeger energies, see Carron--Mondello--Tewodrose \cite[Theorem 1.17]{CMT} and the references therein. 

It is worth pointing out that Theorem \ref{th:MainB} does not require the bounded interpolation property to be uniform over $\mathcal{M}(n,T,\delta)$. As shown by Gigli--Nobili \cite{GigliNobili}, the uniform bounded interpolation property is stable under pmGH convergence. It also implies an $L^1$-Poincaré inequality, as follows from Rajala's work \cite{Rajala}. However, the resulting constant in the latter inequality is a priori neither uniform nor dependent on curvature bounds.

The classes of spaces identified in Theorem \ref{th:MainB} are compact under the following circumstances. If the taming condition $\BE(-\kappa,n)$ hypothesized in Theorem \ref{thm:MainA} is stable under pmGH convergence, then the class $\mathcal{M}(n,T,\delta)$ becomes compact, provided we impose a uniform bound on the reference measures. In addition, we expect that a stronger condition on $\k_T$ implies continuity of the Cheeger energy, hence compactness of $\mathcal{M}(n,T,\delta,\eta)$.

\subsubsection*{Acknowledgments.} M.~Braun acknowledges financial support by the EPFL through a Bernoulli Instructorship. C.~Rigoni's research was funded by the Austrian Science Fund (FWF) [Grants: DOI 10.55776/EFP6, DOI:10.55776/ESP224]. C.~Rose gratefully acknowledges support by the DFG (grant no.: 540199605). D.~Tewodrose is funded by the Research Foundation – Flanders (FWO) via the Odysseus II programme no.~G0DBZ23N. M.~Braun and C.~Rose thank Vrije Universiteit Brussel for their hospitality during the finalization of the article. 

\section{Preliminaries}

In this paper, a metric measure space is a triple $(M,\dist,\m)$ where $(M,\dist)$ is a complete and separable metric space endowed with a fully supported Radon measure $\m$. We denote by $C(M,\dist)$ the space of $\dist$-continuous functions on $M$, by $\Lip(M,\dist)$ the space of all real-valued $\dist$-Lipschitz functions on $M$ and by
$\Lip_b(M,\dist)$,  resp.~$\Lip_c(M,\dist)$, the subspace of bounded, resp.~compactly supported, Lipschitz functions. For any $p \in [1,+\infty]$ we write $L^p(M,\m)$ for the Lebesgue space with exponent $p$ on $(M,\m)$, and $L^p_{\text{loc}}(M,\m)$ for the corresponding space of locally $p$-integrable functions. By $\mathcal{P}(M)$ we mean the space of probability measures on $(M,\dist)$. The indicator function of a set $A \subset M$ is denoted by $\One_A$. We shall often use the shorthand
\[
V(x,r) \df \m(B(x,r))
\]
where $B(x,r)$ is the open ball centered at $x$ with radius $r$. For this section, we consider a fixed metric measure space $(M,\dist,\m)$.\\

\paragraph{Length structure and absolutely continuous curves} We denote by $\contpath$ the space of $\dist$-continuous maps from $[0,1]$ to $M$. For any given $x,y \in M$, we let $\contpathxy$ be the space of  $\gamma \in \contpath$ such that $\gamma_0=x$ and $\gamma_1=y$. Recall that $(M,\dist)$ is called length if for any $x,y \in M$,
\[
\dist(x,y) = \inf \{ L(\gamma) : \gamma  \in \contpathxy \},
\]
where $L(\gamma)$ is the length of $\gamma$, defined as the usual supremum over partitions of $[0,1]$.

Let $\accurves$ denote the set of $L^2$-absolutely continuous curves in $M$. These are the elements $\gamma \in \contpath$ for which there exists $m \in L^2([0,1])$ such that
\begin{equation*}  
\dist(\gamma_t,\gamma_s) \le \int_t^s m(\tau) \d \tau
\end{equation*}
for any $t<s$ in $[0,1]$.
For such a map $\gamma$, the number
\[
|\dot{\gamma}_t| \df \lim_{s \to t} \frac{\dist(\gamma_s,\gamma_t)}{|t-s|}
\]
is a well-defined real number for a.e.~$t \in [0,1]$, and the function $|\dot{\gamma}|$  belongs to $L^2([0,1])$. This function is called the metric slope of $\gamma$ and it is the a.e.~minimal function $m$ satisfying the inequality above. We refer to \cite[Theorem 1.1.2]{AGS-book} for more details.

For any $x, y \in M$, we write $\accurvesxy$ for the set of $\gamma \in \accurves$ satisfying $\gamma_0=x$ and $\gamma_1=y$. Then the following holds.

\begin{lemma}
    The metric space $(M,\dist)$ is length if and only if for any $x,y \in M$,
\[
\dist(x,y) = \inf \left\{\int_0^1 |\dot{\gamma}_t| \d t : \gamma \in \mathrm{AC}_{xy}^2(M,\dist)\right\}.
\]
\end{lemma}

\paragraph{Test plans and  minimal weak upper gradient} Recall some classical definitions due to Ambrosio--Gigli--Savaré \cite[Section 5.1]{AGSInv}. For any $t \in [0,1]$, we write $e_t : \mathcal{C}([0,1],M,\dist) \to M$ for the evaluation map at time $t$ defined by $e_t(\gamma)=\gamma(t)$. Endow the space $\mathcal{C}([0,1],M,\dist)$ with the distance of uniform convergence. Consider a Borel measure  $\mu$ on $(M,\dist)$, possibly different from $\m$. 

\begin{definition}
\begin{enumerate}
    \item We say that $\pi \in \mathcal{P}(\mathcal{C}([0,1],M,\dist))$ is a test plan relative to $\mu$ if
    \begin{enumerate}
        \item[(i)] $\pi(\accurves)=1$ and $\int \int_0^1 |\dot{\gamma}|^2(t) \d t \d \pi(\gamma) < +\infty$,
        \item[(ii)] there exists $C>0$ such that $(e_t)_\# \pi \le C \mu$ for any $t \in [0,1]$.
    \end{enumerate}

\item We say that $g \in L^2(M,\mu)$ is a weak upper gradient relative to $\mu$ of a Borel function $f:M \to \mathbb{R}$ if, for any test plan $\pi$ relative to $\mu$,
    \[
    |f(\gamma_1)- f(\gamma_0)| \le \int_0^1 g(\gamma(s)) |\dot{\gamma}|(s) \d s \qquad \text{for $\pi$-a.e.~$\gamma$.} 
    \]
    
\item The function $g$ as above with  minimal $L^2$ norm is called the minimal weak upper gradient relative to $\mu$ of $f$. We denote it by $|df|_\mu$.  If $\mu=\m$ we simply write $|df|$.

 \item The space $\mathcal{S}^2(M,\dist,\mu)$  of Borel functions on $M$ admitting a weak upper gradient relative to $\mu$ in $L^2(M,\mu)$ is called the Sobolev class of $(M,\dist,\mu)$.
 \end{enumerate}
\end{definition}

We will use a variant of \cite[Proposition 5.7]{AGSInv} that was used in \cite{GarofaloMondino-13}.

\begin{lemma}
Consider $f:M \times [0,1] \to \mathbb{R}$ such that $f(\cdot,t) \in  \mathcal{S}^2(M,\dist,\mu)$ for a.e.~$t\in (0,1)$. Then for any differentiable function $\alpha : [0,1] \to [0,1]$ and any test plan $\pi$ relative to $\mu$, the function $s \mapsto f(\gamma_s,\alpha_s)$ belongs to $W^{1,1}([0,1])$ for $\pi$-a.e.~$\gamma$, with
\[
\left| \frac{\d}{\d s} f(\gamma_s,\alpha_s)\right| \le |d[f(\cdot,\alpha_s]|_\mu (\gamma_s) |\dot{\gamma}_s| + |\dot\alpha_s| \left| \frac{\partial f}{\partial t} (\gamma_s,\alpha_s) \right|\qquad \text{for a.e.~$s\in (0,1)$.}
\]
\end{lemma}

\paragraph{Bounded interpolation property} Denote by $\mathrm{Geo}(M,\dist)$ the space of $\dist$-geodesics on $M$. These are the elements $\gamma \in \cC([0,1],M,\dist)$ such that $\dist(\gamma(s),\gamma(t)) = \dist(\gamma(0),\gamma(1)) |t-s|$ for any $s,t \in [0,1]$. We endow $\mathrm{Geo}(M,\dist)$ with the infinite norm topology inherited from $\cC([0,1],M,\dist)$.  For any $q \in (1,+\infty)$ and $\mu_0,\mu_1 \in \mathcal{P}(M)$ with finite $q$-moment, we let $\mathrm{Opt}_q(\mu_0,\mu_1)$ be the set of optimal couplings between $\mu_0$ and $\mu_1$ for the $q$-Wasserstein distance, and we define the set of dynamical optimal plans between $\mu_0$ and $\mu_1$ as
\[
\mathrm{OptGeo}_q(\mu_0,\mu_1) \df \{\pi \in \mathcal{P}(\mathrm{Geo(M,\dist)}) : (e_0,e_1)_\# \pi \in \mathrm{Opt}_q(\mu_0,\mu_1) \}.
\]

The bounded interpolation property  from \cite{GigliNobili} writes as follows.

\begin{definition}\label{Def:BIP}
    We say that $(M,\dist,\m)$ satisfies the bounded interpolation property (or is BIP, for brevity) if for any $q \in (1,+\infty)$ there exists $C_q : (0,+\infty) \to [1,+\infty)$ such that for any $\mu_0,\mu_1 \in \mathcal{P}(M)$ that are absolutely continuous with respect to $\m$ with bounded densities $\rho_0,\rho_1$, if $\mathrm{diam}(\mathrm{supp}(\mu_0) \cup \mathrm{supp}(\mu_1) ) < D$ for some $D>0$, then there exists $\pi \in \mathrm{OptGeo}_q(\mu_0,\mu_1)$ such that for any $t \in [0,1]$,
    \[
    (e_t)_\# \pi  \ll \m \quad \text{with} \quad \left\| \rho_t \right\|_\infty \le C_q(D)   \max \left( \left\| \rho_0\right\|_\infty,\left\| \rho_1 \right\|_\infty \right)
    \]
    where $\rho_t$ is the density of $(e_t)_\# \pi$ with respect to $\m$.
\end{definition}

\paragraph{Infinitesimally Hilbertian spaces} The Cheeger energy of the space $(M,\dist,\m)$ is the functional $\Ch:\Ltwo\to [0,\infty]$ defined by setting
\[\Ch (f) :=\int_M|d f|^2\d\m \quad \text{if $f \in \mathcal{S}^2(M,\dist,\mu),$} \qquad \text{ $\Ch (f)=+\infty$ otherwise.}\] 
 It is convex and lower semicontinuous. The associated Sobolev space is defined as
\[
W^{1,2}(M,\dist,\m) \df L^2(M,\m) \cap \mathcal{S}^2(M,\dist,\m) = \{f \in \Ltwo \, : \, \Ch(f) < +\infty\}.
\]
It is a Banach space when endowed with the norm
$$
\|f\|_{W^{1,2}}:=\left(\|f\|_{L^2}^2+\Ch(f)\right)^{1/2}.
$$
The next seminal definition is due to Gigli \cite{Gigli15}.

\begin{definition}
A metric measure space $(M,\dist,\m)$ is called 
infinitesimally Hilbertian if its Cheeger energy is quadratic or, equivalently, if the Sobolev space $(\Sob,\|\cdot\|^2_{W^{1,2}})$ is Hilbert.
\end{definition}

Infinitesimal Hilbertianity provides access to the well-established theory of Dirichlet forms. Indeed, from \cite[Theorem 4.18]{AGSDuke}, it follows that if $(M,\dist,\m)$ is infinitesimally Hilbertian, then
$$
\Gamma : ( f_1, f_2) \mapsto \lim_{\epsilon\to 0}\frac{|d (f_1+\epsilon f_2)|^2-|d f_1|^2}{2\epsilon}
$$
is a well-defined symmetric bilinear form on $\Sob$ with values in $L^1(\M,\dist,\m)$, with associated quadratic form still denoted by $\Gamma$. Of course, $$\Gamma(f)= |df|^2$$ $\m$-a.e.~on $M$ for any $f \in \Sob$. Moreover,
$$
\Ch(f_1,f_2) := \int_M \Gamma ( f_1,  f_2 ) \d \m
$$
defines a strongly local Dirichlet form on $(M,\dist,\m)$.  This has the following notable consequences. 

\begin{enumerate}
    \item The space $(M,\dist,\m)$ has a linear Laplace operator $\Delta$. This is the operator canonically associated with the Dirichlet form $\Ch$, with dense domain $\D(\Delta)$, defined by
\begin{align*}
    \Ch(u,v) = -\int_M u\,\Delta v\d\m
\end{align*}
for any $u \in \Sob$ and $v \in \D(\Delta)$. Note that here we take the convention that $\Delta$ is non-positive.
\item The space $(M,\dist,\m)$ has a heat semigroup $(\P_t)_{t > 0}$ defined as the self-adjoint sub-Markovian semigroup acting on $\Ltwo$ associated with $\Ch$. This semigroup extends to a positivity preserving, $\m$-symmetric, contraction semigroup on $L^p(\M,\dist,\m)$ for any $p \in [1,+\infty]$, which is strongly continuous when $p < +\infty$. In particular, for any $f \in L^\infty(\M,\dist,\m)$ and $t > 0$,
\[
\Delta \P_tf = \frac{\partial}{\partial t} \P_tf.
\]

\item The space $(M,\dist,\m)$ is upper regular in the sense of \cite[Definition 3.13]{AGSAnnProb}, see \cite[Section 8.3]{AGSIber}. This property ensures that the intrinsic distance associated with the Cheeger energy, defined by
\[
\dist_{\Ch} (x,y) \df \sup \{|f(x)-f(y)| : f \in \Sob \text{ with } \Gamma(f)\le 1 \text{ $\m$-a.e.} \}
\]
for any $x,y \in M$,
coincides with the initial distance $\dist$.
\item If there exists a family of $(\m\otimes\m)$-measurable functions $(p_t(\cdot,\cdot))_{t>0}$ on $M\times M$ such that for all $t>0$ and $f \in \Ltwo$, one has 
\[
\P_tf(x) = \int_Mp_t(x,y)f(y) \d \m(y) \qquad \text{for $\m$-a.e.~$x \in M$},
\]
then the function $p$ is called the heat kernel of $(M,\dist,\m)$. However we remark that the existence of such a heat kernel is not automatic. 
\end{enumerate}

\paragraph{Singular perturbation theory} For an infinitesimally Hilbertian  space $ (M,\dist,\m)$, we denote by $\mathcal{M}_0(M,\dist,\m)$ the space of non-negative Borel measures on $(M,\dist)$ which do not charge Cheeger-polar sets. The latter point means that any $\kappa \in \mathcal{M}_0(M,\dist,\m)$ is absolutely continuous with respect to the $2$-capacity associated with the Cheeger energy.

\begin{definition}\label{def:Kato}
For  an infinitesimally Hilbertian  space $ (M,\dist,\m)$, $\kappa \in \mathcal{M}_0(M,\dist,\m)$ and $t>0$, set
\begin{align*}
\k_T(\M,\dist,\m,\kappa)\df\text{q-}\sup_{x \in M} \mathbb{E}_x[A_t],
\end{align*}
where $A_t$ is the positive continuous additive functional associated with $\kappa$ (see \cite[Theorems 4.2.6 and 6.7.2]{CF} and \cite[Lemma 2.11]{ERST}) and q-$\sup$ is the quasi everywhere supremum associated with $\Ch$. Then the extended Kato class of $(M,\dist,\m)$ with parameters $T,c>0$ is defined as
\[
\K^-_{T,c}(\M,\dist,\m) \df \{ \kappa \in \mathcal{M}_0(M,\dist,\m) : \k_T(M,\dist,\m,\kappa)<c \}.
\]
\end{definition}

Measures in the extended Kato class of  an infinitesimally Hilbertian space define perturbed Cheeger energies, as given by the following result.

\begin{proposition}
    Let $ (M,\dist,\m)$ be an infinitesimally Hilbertian  space. Consider $\kappa \in \K^-_{T,c}(\M,\dist,\m)$ for some $T,c>0$. Then the quadratic form $\Ch^\kappa \df \Ch-\kappa$ is a closed, lower semibounded quadratic form.
\end{proposition}

In the following, we will denote by $\Delta^\kappa$ the generator of $\Ch^\kappa$.

\paragraph{Taming} In \cite{ERST}, the notion of a Dirichlet space with synthetic distributional Ricci curvature lower bounds was given. In particular, when the distribution is a measure and the Dirichlet form is the Cheeger energy of an infinitesimally Hilbertian metric measure structure, this notion writes as follows.

\begin{definition}\label{Def:BE}
Let $ (M,\dist,\m)$ be an infinitesimally Hilbertian space. For any $\kappa \in \mathcal{M}_0(M,\dist,\m)$  and $n \in [1,+\infty]$, we say that $(M,\dist,\m)$ satisfies $\mathrm{BE}(-\kappa,n)$ if
\begin{equation*}
\Gamma_2[f,\phi] \ge \frac{1}{n} \int_M \phi \,  (\Delta f)^2 \d \m - \int_M \phi  \Gamma(f) \d \kappa,
\end{equation*}
for every $f \in \D(\Delta)$ with $\Delta f \in \Sob$, and for every $\m$-a.e.~non-negative $\phi \in \D(\Delta^{2\kappa})$ with $\phi,\Delta^{2\kappa} \phi \in \Linf$, where
\begin{align*}
    \Gamma_2[f,\phi] := -\frac{1}{2}\int_\M\Gamma(\Gamma(f),\phi)\d\m - \int_\M\Gamma(f,\Delta f)\,\phi\d\m.
\end{align*}

\begin{remark}
    If $\kappa$ is absolutely continuous with respect to $\m$ with Radon--Nikodym derivative $\m$-a.e.~equal to a constant $K \in \mathbb{R}$, then the inequality defining $\BE(-\kappa,n)$ rewrites as the classical Bakry--\'Emery curvature-dimension condition \begin{equation*}
\Gamma_2[f,\phi] \ge  \int_M \phi  \left( \frac{(\Delta f)^2}{n}- K\Gamma(f)\right)\d \m.
\end{equation*}
\end{remark}
\end{definition}

\section{From taming to Bakry-\'Emery}

In this section, we show that an infinitesimally Hilbertian space satisfying the taming condition from Definition~\ref{Def:BE} with a suitable measure $\kappa$ is bi-Lipschitz equivalent to a metric measure space satisfying the Bakry--\'Emery curvature-dimension condition $\BE(K,N)$ for some $K \in \mathbb{R}$ and $N\ge n$.

\subsection{Preliminaries on time changes}

We refer to \cite[Chapter 5]{CF} for the definition and main properties of time changes in the general setting of Dirichlet spaces. Here we focus on the case of infinitesimally Hilbertian length spaces, which allows for some simplifications, see e.g.~\cite{SturmGAFA,hansturm}. For such a space $(M,\dist,\m)$, the space of local Sobolev functions is defined as
\[
\Sobloc \df \{f \in \Ltwoloc : \varphi f \in \Sob \text{ for any } \varphi\in \Lip_c(M,\dist)\}.
\]

\begin{definition}
Let $(M,\dist,\m)$ be an infinitesimally Hilbertian length space. For any $w \in L^\infty_{\loc}(M,\m)$, the associated time-change metric measure space $ (M,\dist_w,\m_w)$ is defined by
\begin{align*}
\dist_w(x,y) & := \sup \left\{ \varphi(x) - \varphi(y) : \varphi \in \Sobloc \cap C(M,\dist) \text{ s.t.~} \Gamma(f) \le e^w \, \m\right\}
\end{align*}
for any $x, y \in M$, and 
$$\m_w  := e^{2w} \m.$$
\end{definition}

In the rest of the paper, all objects associated with $(M,\dist_w,\m_w)$ will be written with a sub- or superscript $w$ without further notice.

\begin{remark}
If the function $w$ is continuous $\m$-a.e., then the metric $\dist_w$ admits the following dual representation: for any $x,y \in M$,
\begin{align*}
\dist_w(x,y) & := \inf \left\{\int_0^1 e^{\bar{w}(\gamma(t))}|\dot{\gamma}_t| \d t : \gamma \in \mathrm{AC}_{xy}(M,\dist)\right\},
\end{align*}
where $\bar{w}$ is the upper semicontinuous envelope of $w$.
\end{remark}

The next proposition collects several elementary properties of time-change spaces in our setting.

\begin{proposition}
Let $(M,\dist,\m)$ be an infinitesimally Hilbertian length space. Consider $w \in L^\infty_{\text{loc}}(M,\m)$  continuous $\m$-a.e.~on $M$.

\begin{enumerate}
\item The space $(M,\dist_w,\m_w)$ is an infinitesimally Hilbertian length metric measure space that satisfies upper regularity.

\item The vector spaces $W^{1,2}(M,\dist_w,\m_w)$ and $\Sob$ coincide, and for any $f \in W^{1,2}(M,\dist_w,\m_w)$,
\begin{equation*}
    \Gamma^w(f) = e^{-2w} \Gamma(f) \qquad \text{$\m$-a.e.~on $M$.}
    \end{equation*}
   
\item The Cheeger energies $\Ch^w$ and $\Ch$ coincide.

\item For any $f \in \D(\Delta)\cap L^\infty(\M,\m)$ such that $\Delta f \in \Sob \cap L^\infty(\M,\m)$, and $\m$-a.e.~non-negative $\phi \in \Sob\cap L^\infty(\M,\m)$, it holds
\begin{align*}
\Gamma_2^w[f,\phi]  
& =  \Gamma_2[f,e^{-2w}\phi]  + \int_M \Gamma(w,\Gamma(f)e^{-2w} \phi) \d \m \\
& - 2 \int_M\big( \Gamma(\Gamma(f),w) - \Gamma(w)\Gamma(f) - \Gamma(f,w)\Delta f \big) e^{-2w} \phi \d \m.
\end{align*}
\end{enumerate}
\end{proposition}

We shall also need the following result.

\begin{lemma}\label{lem:HKtimechange}
    Let $(M,\dist,\m)$ be an infinitesimally Hilbertian length space admitting a heat kernel $p$. Consider $w \in L^\infty(\M,\m)$  continuous $\m$-a.e.~on $M$. Then $(M,\dist_w,\m_w)$ admits a heat kernel $p^w$ as well.
\end{lemma}

\begin{proof}
By \cite[Lemma A.3.10]{CF}, the resolvent of the time-changed Dirichlet form is given by
\begin{align*}
    R_\alpha f(x) = \mathbf{E}^x\left[ \int_0^\infty e^{-\alpha A_t}\,f(X_t)\d A_t \right].
\end{align*}
Since the positive continuous additive functional is given by $A_t = \int_0^t w(X_s)\d s$, the previous term writes as
\begin{align*}
    R_\alpha f(x) = \mathbf{E}^x \left[\int_0^\infty e^{-\alpha \int_0^t w(X_s)\d s} f(X_t)\,w(X_t)\d t \right].
\end{align*}
Now, define a set function $B\mapsto R_\alpha 1_B(x)$. Let $B$ with $\m_w[B]=0$ (hence $\m[B]=0$). Then
\begin{align*}
    R_\alpha 1_B(x) = \mathbf{E}^x\left[\int_0^\infty e^{-\alpha\int_0^t w(X_s)\d s} 1_B(X_t)\,w(X_t)\d t\right] & \leq C \mathbf{E}^x\left[\int_0^\infty e^{\alpha Ct}\,1_B(X_t)\d t\right] \\
    & = C\int_0^\infty e^{\alpha Ct}\,\P_t1_B(x)\d t=0
\end{align*}
because
\begin{align*}
    \P_t1_B(x) = \int_Bp_t(x,y)\d \m(y) = 0.
\end{align*}
This means that the resolvent is absolutely continuous with respect to $\m_w$, and therefore, the heat kernel  $p^w$ exists.
\end{proof}

For an infinitesimally Hilbertian space $(M,\dist,\m)$, we define the following sets of test functions:
\begin{align*}
\Test(\M,\dist,\m) &:= \{f\in \D(\Delta) \cap L^\infty(\M,\m) : \Gamma(f) \in L^\infty(\M,\m),\,\Delta f\in \Sob\},\\
\Test_{c}(\M,\dist,\m) &:= \{f\in\Test(\M) \text{ with compact support}\}\\
\Test_\infty(\M,\dist,\m) &:= \{f\in\Test(\M) : \Delta f \in L^\infty(\M,\m)\}.
\end{align*}
They are all dense in $\Sob$, hence in $\Ltwo$, relative to the respective topologies.

The next result expresses how the taming condition is transformed under time-change. The proof is essentially the same as that for its predecessor \cite[Thm.~4]{hansturm}. Our hypotheses on the weight are more general, as we integrate by parts once less.

\begin{proposition}\label{thm:timechange} Let $(M,\dist,\m)$ be an infinitesimally Hilbertian length space. For $T,c>0$ and $n \in [1,+\infty)$, assume there exists $\kappa \in \K^-_{T,c}(\M,\dist,\m)$ such that $(M,\dist,\m)$ satisfies $\BE(-\kappa,n)$. Consider $w\in L^\infty(\M,\m)$ such that $\Gamma(w)\in L^1_\loc(\M,\m)$, and $N \ge 1$ such that  $\hat{N} \df \frac{(n-2)(N-2)}{N-n}$ is non-negative.
Then
\begin{align*}
    \Gamma_2^w[f,\phi]  \geq & - \hat{N}\int_\M \,\Gamma(w)\,\Gamma^w(f)\,\phi\d\m + \int_\M\Gamma(w,\Gamma^w(f)\,\phi)\d\m \\
    &   - \int_\M\Gamma^w(f)\,\phi\d \kappa + \frac{1}{N}\int_\M(\Delta^wf)^2\,\phi\d\m_w,
\end{align*}
for any $f\in\Test_\infty(\M)$ and any $\m$-a.e.~non-negative $\phi\in\Test_c(\M)$.
\end{proposition}

\subsection{Main result}
We provide now the main result of this section, which extends \cite[Theorem 1.1]{CMT_KatoBE} to tamed spaces. For any $n \in [1,+\infty)$, we set $$c(n)\df 1\wedge \frac{1}{2|n-2|}\, \cdot $$

\begin{theorem}\label{thm:BE}Let $(M,\dist,\m)$ be an infinitesimally Hilbertian length space satisfying $\m(M)<+\infty$. For $T>0$ and $n \in [1,+\infty)$, assume there exists $\kappa \in \K^-_{T,c(n)}(M,\dist,\m)$ such that $(M,\dist,\m)$ satisfies $\mathrm{BE}(-\kappa,n)$. Then there exists an $\m$-a.e.~continuous function $w\in \Sob$ satisfying $0\leq w\leq 4\k_T(\M,\dist,\m,\kappa)$ $\m$-a.e.~such that $(\M,\dist_w,\m_w)$ satisfies $\mathrm{BE}(-K,N)$ with
\[
K \df \frac{8\k_T(\M,\dist,\m,\kappa)}{T} \qquad \text{and} \qquad N \df n+3(1\vee (n-2)^2)\k_T(\M,\dist,\m,\kappa).
\]
\end{theorem}

In order to prove this result, we need two preliminary lemmas.

\begin{lemma}\label{lem:gaugefunction} 
    Let $(M,\dist,\m)$ be an infinitesimally Hilbertian length space satisfying $\m(M)<+\infty$. For $T>0$ and $n \in [1,+\infty)$, assume there exists $\kappa \in \K^-_{T,c(n)}(M,\dist,\m)$ such that $(M,\dist,\m)$ satisfies $\mathrm{BE}(-\kappa,n)$. Suppose there exists $\beta>0$ such that
\begin{align*}
    \k_T(M,\dist,\m,\kappa) \leq 1-\e^{-\beta T} < c(n).
\end{align*}
Then there exists an $\m$-a.e.~continuous function $v \in \D(\Delta^{\kappa})$ satisfying $\smash{1\leq v\leq 2\,\e^{\beta T}}$ $\m$-a.e.~and
\begin{align*}
    -\Delta^{\kappa}v +2\,\beta\,v = 2\,\beta \qquad \text{$\m$-a.e.}
\end{align*}
\end{lemma}

\begin{proof}
Set
\[
v := 2\beta \int_0^\infty \e^{-2\beta t}\,\P_t^{\kappa}\One_\M\d t.
\]
Since $\m(\M)<+\infty$, the heat semigroup $(\P_t)_{t\geq 0}$ is conservative, i.e., $\P_t\One_\M=\One_\M$, $t\geq0$.
By the Feynman-Kac formula, we have $$\P_t^{\kappa}\One_\M\geq \P_t\One_\M=\One_\M$$ for any $t>0$, so that $v\geq1$ $\m$-a.e. In order to estimate the supremum norm, note that the Myadera-Voigt perturbation theorem \cite{Voigt-77} implies 
\[
\|\P_t^{\kappa}\|_{\infty,\infty}\leq \e^{\beta (T+t)}, \qquad t>0,
\]
where $\|A\|_{\infty,\infty}$ denotes the norm of the bounded operator $A$ in $L^\infty(\M,\m)$.
Hence, we obtain
\begin{align*}
\|v\|_\infty
&\leq 2\beta\int_0^\infty e^{-2\beta t}\| \P_t^{\kappa}\One_\M\|_\infty \d t
\\
&\leq 2\beta\int_0^\infty \e^{-2\beta t}\| \P_t^{\kappa}\|_{\infty,\infty} \d t 
\\
&\leq 
2\beta \int_0^\infty \e^{\beta(T- t)} \d t
\\
&=2\e^{\beta T}.
\end{align*}
Since $\m(X)<\infty$, we have $\One_\M\in L^2(X,\m)$ and therefore $v\in\D(\Delta^{\kappa})\subset \Sob$, where the last inclusion follows since $\kappa$ is in the extended Kato class. Hence, $v$ solves the desired equation $\m$-almost everywhere. By \cite[Theorem~2.1.3]{FukushimaYoshidaTakeda}, $v \in \Sob$ possesses a quasi-continuous modification, which therefore is $\m$-almost everywhere continuous.
\end{proof}
\begin{remark}
 If $M$ is a smooth complete Riemannian manifold and $\kappa$ is absolutely continuous with respect to the Riemannian volume measure with continuous density, then Lemma~\ref{lem:gaugefunction} follows from elliptic regularity theory.
\end{remark}

We are now in a position to prove Theorem \ref{thm:BE}.

\begin{proof} Set $\k:=\k_T(M,\dist,\m,\kappa)$. If $\k=0$, then $(\M,\dist,\m)$ satisfies the Bakry--\'Emery curvature-dimension condition $\mathrm{BE}(0,n)$, and the statement is trivial. Thus, we assume that $\k>0$.
Set \[
\hat N \df \frac{(n-2)(N-2)}{N-n},\quad
\lambda \df \begin{cases}
    \frac{1}{2\k}& \colon n= 2,\\
    \hat N&\colon n\neq 2,
\end{cases}
\quad\text{and}\quad \beta \df  \frac{1}T\ln\left(\frac1{1-\lambda\k} \right)>0.
\]
As the definition of $\beta$ implies $
\k= (1-\e^{-\beta T})/\lambda$, we have 
\[
\k_T(M,\dist,\m,\lambda\kappa)=\lambda\k= 1-\e^{-\beta T}.
\]
By Lemma~\ref{lem:gaugefunction}, there exists a quasi-continuous $v\in \D(\Delta^{\lambda \kappa})\subset \Sob$ such that 
$\smash{1\leq v\leq 2\,\e^{\beta T}}$ and $- \Delta^{\lambda\kappa}v +2\,\beta\,v = 2\,\beta$ $\m$-a.e. In particular, we have $\Gamma(v)\in L^1(M,\m)$. We set $$w:=\frac{1}{\lambda}\ln v.$$ 
Then $0\leq w\leq \frac{\ln(2)+\beta T}{\lambda}$ $\m$-a.e. Therefore, Proposition~\ref{thm:timechange} yields that, for every $f\in\Test_\infty(\M)$ and every $\m$-a.e.~non-negative $\phi\in\Test_c(\M)$,
\begin{align*}
    \Gamma_2^w[f,\phi]  \geq & - \hat{N}\int_\M \,\Gamma(w)\,\Gamma^w(f)\,\phi\d\m + \int_\M\Gamma(w,\Gamma^w(f)\,\phi)\d\m \\
     & - \int_\M\Gamma^w(f)\,\phi\d \kappa + \frac{1}{N}\int_\M(\Delta^wf)^2\,\phi\d\m_w.
\end{align*}
In order to obtain the desired Bakry-\'Emery curvature-dimension condition, we are left with bounding from below the sum of the first three integrals in the previous right-hand side.

Expanding the first integral using the definition of $w=\frac{1}{\lambda}\ln v$ yields 
\[
    - \hat{N}\int_\M \,\Gamma(w)\,\Gamma^w(f)\,\phi\d\m
    =- \frac{\hat{N}}{\lambda}\int_\M \,\frac{\Gamma(v)}{v}\,\frac{\Gamma^w(f)\,\phi}{\lambda v}\d\m \text.
\]

For the second integral, we first claim $\Gamma^w(f)\phi = e^{-2w}\Gamma(f)\phi\in W^{1,2}(M,\dist,\m)\cap L^\infty(M,\m)$. Boundedness is straightforward since each factor belongs to $L^\infty(\M,\m)$, where this property for $\smash{e^{-2w}}$ follows from \Cref{lem:gaugefunction}. Since bounded Sobolev functions form an algebra, it remains to show that each factor belongs to $W^{1,2}(\M,\dist,\m)$. For $\phi$ and $\Gamma(f)$, this follows directly from the definition of $\Test(\M,\dist,\m)$. The inclusion $e^{-2w}\in W^{1,2}(\M,\dist,\m)$ follows from the chain rule, finiteness of $\m$, and the property $w\in W^{1,2}(\M,\dist,\m)\cap L^\infty(\M,\m)$ implied by \Cref{lem:gaugefunction}. Therefore, the Leibniz rule gives
\begin{multline*}
    \int_\M\Gamma(w,\Gamma^w(f)\,\phi)\d\m
    =\int_\M\frac{1}{\lambda v}\Gamma(v,\Gamma^w(f)\,\phi)\d\m 
    \\=
    \int_\M\Gamma\left( v,\frac{\Gamma^w(f)\,\phi}{\lambda v}\right)\d\m 
     -\int_\M \Gamma\left( v,\frac{1}{\lambda v}\right)\Gamma^w(f)\phi\d\m
    \\
    =\int_\M\Gamma\left( v,\frac{\Gamma^w(f)\,\phi}{\lambda v}\right)\d\m 
     +\int_\M \frac{\Gamma(v)}{v}\frac{\Gamma^w(f)\phi}{\lambda v}\d\m.
\end{multline*}
Using $\lambda\geq \hat N$ and $\Gamma(v)/v\geq 0$ $\m$-a.e.~to drop the first integral in the second line, we obtain
\begin{multline*}
    - \hat{N}\int_\M \,\Gamma(w)\,\Gamma^w(f)\,\phi\d\m + \int_\M\Gamma(w,\Gamma^w(f)\,\phi)\d\m
     - \int_\M\Gamma^w(f)\,\phi\d \kappa
     \\
     =
    \left(1 - \frac{\hat{N}}{\lambda}\right)\int_\M \,\frac{\Gamma(v)}{v}\,\frac{\Gamma^w(f)\,\phi}{\lambda v}\d\m 
     + \int_\M\Gamma\left( v,\frac{\Gamma^w(f)\,\phi}{\lambda v}\right) \d\m 
    - \int_\M \frac{\Gamma^w(f)\,\phi}{\lambda v} \lambda v\d \kappa
     \\
     \geq 
   \int_\M\Gamma\left( v,\frac{\Gamma^w(f)\,\phi}{\lambda v}\right) \d\m 
    - \int_\M \frac{\Gamma^w(f)\,\phi}{\lambda v} \lambda v\d \kappa
    =\int \frac{\Gamma^w(f)\phi}{\lambda v}\left[-\Delta^{-\lambda\kappa}v\right]\d \m.
\end{multline*}
Since $ -\Delta^{\lambda\kappa}v=2\beta-2\beta v=2\beta(1-v)$ and $1\leq v\leq 2e^{\beta T}$ $\m$-a.e., we obtain, using $\frac{\Gamma^w(f)\phi}{\lambda v}\in \Test(\M)$ 
and $\frac{\Gamma^w(f)\phi}{\lambda v}\geq 0$ $\m$-a.e.,
\[
\int \frac{\Gamma^w(f)\phi}{\lambda v}\left[-\Delta^{-\lambda\kappa}v\right]\d \m
=
\frac{2\beta}\lambda\int \Gamma^w(f)\phi\left[\frac{1}v-1\right]\d \m
\geq -\frac{\beta}\lambda (2-e^{-\beta T})
\int \Gamma^w(f)\phi\d \m.
\]
We are left with estimating the constant on the right-hand side. 
Since  $\lambda \k=1-e^{-\beta T}$ and $\beta= \frac{1}T\ln\left(\frac1{1-\lambda\k} \right)$ we have
\[
\frac{\beta}\lambda (2-e^{-\beta T})
= \frac{1}T\ln\left(\frac1{1-\lambda\k} \right)\left(\k+\frac{1}{\lambda}\right).
\]
In order to estimate the right-hand side, we  distinguish three cases. Recall $\hat N=\tfrac{(n-2)(N-2)}{N-n}$.
\\
\textbf{Case 1:} $n>2$. We have $N=n+3(n-2)^2\k$, $\lambda=\hat N$, and $\k< 1\wedge \tfrac{1}{2(n-2)}$, so that 
\[
\lambda \k=\hat N\k
=
\frac{1}{3}+(n-2)\k\leq \frac{5}{6},\qquad \frac{1}{\lambda}=\frac{3(n-2)^2\k}{(n-2+3(n-2)^2\k)(n-2)}\le 3\k.
\]
\textbf{Case 2:} $n<2$. We have $N=n+2(n-2)^2\k$, $\lambda=\hat N$, and $\k<1\wedge \tfrac{1}{2(2-n)}$, so that, since $n-2<0$,
\[
\lambda \k=\hat N\k
=
\frac{1}{2}+(n-2)\k\leq \frac{1}{2},\qquad \frac{1}{\lambda}=\frac{2(2-n)^2\k}{(2(n-2)^2\k+(2-n))(2-n)}\le 2\k.
\]
\textbf{Case 3:} $n=2$. We have $N=n+2\k$, $\lambda=1/2\k$, and $\k< 1$, such that
\[
\lambda \k=\frac{1}{2},
\qquad \frac{1}{\lambda}= 2\k.
\]
Hence, we always have $\lambda\k\leq 5/6$ and $1/\lambda\leq 3\k$. Therefore,
\[
 \frac{1}T\ln\left(\frac1{1-\lambda\k} \right)\left(\k+\frac{1}{\lambda}\right)
\leq 
\frac{1}T\ln\left(\frac1{1-5/6} \right)\left(\k+3\k\right)
=\frac{8\k}{T},
\]
and the claim follows.
\end{proof}

\section{Doubling and Poincaré via time-change}

In this section, we prove Theorem \ref{thm:MainA}. To this purpose, we  consider an infinitesimally Hilbertian length  metric measure space $(\M,\dist,\m)$ with finite total measure. We assume that $(\M,\dist,\m)$ admits a heat kernel $p$, that it satisfies the bounded interpolation property, and that there exist $T>0$, $n \ge 1$, and $\kappa \in \K^-_{T,c(n)}(\M,\dist,\m)$ such that $(M,\dist,\m)$ satisfies BE$(-\kappa,n)$.

By Theorem \ref{thm:BE}, there exists $w \in \Sob$ $\m$-a.e.~continuous satisfying $0 \le w \le 4 \k_T(M,\dist,\m,\kappa)$ such that the time-changed space $(M,\dist_w,\m_w)$ satisfies $\mathrm{BE}(-K,N)$ with \[K=\frac{8\k_T(M,\dist,\m,\kappa)}{T}\quad \text{and} \quad  N=n+3(1\vee (n-2)^2) \k_T(M,\dist,\m,\kappa).\] We also know from Lemma \ref{lem:HKtimechange} that $(M,\dist_w,\m_w)$ admits a heat kernel $p^w$.

Let us provide a series of technical results adapted from the paper of Garofalo--Mondino \cite{GarofaloMondino-13}. The context of this paper is a metric measure space satisfying the Riemannian curvature-dimension condition $\RCD^*(\rho,\eta)$ for some $\rho \in \mathbb{R}$ and $\eta \in [1,+\infty)$. From the work of Erbar--Kuwada--Sturm \cite[Theorem 7]{EKS}, it is well-known that the $\RCD^*(\rho,\eta)$ and $\BE(\rho,\eta)$ conditions are equivalent on any infinitesimally Hilbertian length metric measure space with finite measure satisfying the Sobolev-to-Lipschitz property. The latter means that every Sobolev function $f$ admits a Lipschitz representative with Lipschitz constant smaller than the $L^\infty$ norm of the minimal weak upper gradient $|df|$.

In our context, we do not know whether the time-changed space $(M,\dist_w,\m_w)$ satisfies the Sobolev-to-Lipschitz property. However, the proofs in \cite{GarofaloMondino-13} are based on the Bakry--\'Emery curvature-dimension condition $\BE(\rho,\eta)$, hence they remain valid for $(M,\dist_w,\m_w)$.

We begin with the following variant of \cite[Proposition 3.2]{GarofaloMondino-13}. This result was first proved for Riemannian manifolds by Baudoin--Garofalo \cite[Theorem 2.2]{BaudoinGarofalo}, based on a preliminary sub-Riemannian version obtained in \cite[Proposition 6.5]{BaudoinGarofalo2}, see also \cite[Theorem 5.2]{BaudoinGarofalo3}.

\begin{proposition}\label{prop:gauge}
Consider $f\in L^\infty(\M, \m_w)$ with $f \ge \delta$ $\m_w$-a.e.~for some $\delta>0$. For any $0 < s < t$, set 
\[\Phi^w(s)=\P_s^w\bigg((\P_{t-s}^wf)\, |d(\ln \P^w_{t-s}f)|^2_{w}\bigg).\]
    For any $\cC^1$ function $a\colon(0,t)\to \RR_+$ and any $\varphi\in \Linfw$ such that $\varphi \ge 0$ $\m_w$-a.e., the function
    \[
    s\mapsto \int [ a \Phi^w ](s) \, \varphi\d \m_w 
    \]
    is absolutely continuous on $(0,t)$. Moreover, for any $\cC^0$ function $\gamma\colon(0,t)\to \RR$, for a.e.~$s\in(0,t)$,
    \begin{align*}
     \frac{\d}{\d s} \left( \int [ a \Phi^w ](s)\, \varphi\d \m_w  \right) \nonumber 
    & \geq \int \left(\left[a'-\frac{4a \gamma}{N}-2Ka\right] \mkern-5mu(s) \, \Phi^w(s)\right.  \nonumber \\
    &  \qquad +  \left. \left[\frac{4a\gamma}{N}\right]\mkern-5mu(s)\, \Delta^w P_{t}^wf-\left[\frac{2a\gamma^2}{N} \right]\mkern-5mu(s)\, P_{t}^wf\right)\varphi\d \m_w . 
    \end{align*}

\end{proposition}

The next result is a Li-Yau inequality for the time-changed space $(M,\dist_w,\m_w)$. This follows from the argument of  \cite[Theorem~1.3]{GarofaloMondino-13}, using that this space satisfies BE$(-K,N)$ with $K> 0$.
\begin{proposition}\label{prop:LiYau}
    For any $f \in \Lonew$ and $t >0$,
\begin{equation*}
| d (\P_t^wf)|_{w}^2\leq \e^{2Kt/3}(\Delta^w\P^w_tf)\P^w_tf-\frac{N K}{3}\frac{ \e^{4Kt/3}}{1-\e^{2Kt/3}}(\P_t^wf)^2 \quad \text{$\m_w$-a.e.}
\end{equation*}
\end{proposition}

The next statement is a Harnack inequality for the heat flow  of the time-changed space $(\M, \d_w, \m_w)$. We obtain it by following the lines of \cite[Theorem 1.4]{GarofaloMondino-13}. We provide the details since this is where we make use of the BIP property of $(M,\dist,\m)$ to circumvent the fact that $(M,\dist_w,\m_w)$ is not necessarily $\RCD^*(-K,N)$.

\begin{lemma}\label{lem:LYtimechange}
Consider $f\in L^1(\M,\m_w)$. For any $s,t>0$ such that $s<t$, and $\m_w$-a.e.~$y,z\in \M$,
\[
(\P_t^wf)(z)\geq (\P_s^wf)(y)\, \e^{-\frac{\dist_w^2(z,y)}{4(t-s)\e^{-2Kt/3}}}\left(\frac{1-\e^{-2Ks/3}}{1-\e^{-2Kt/3}}\right)^{\frac{N}{2}}.
\]
In particular, the heat kernel $p^w$ of
$(\M,\dist_w,\m_w)$ satisfies, for $\m_w$-a.e.~$x,y,z\in\M$ and $0<s<t$,
\[
p_t^w(y,z)\geq p_s^w(x,z)\e^{-\frac{\dist_w(x,y)^2}{4(t-s)\e^{2Kt/3}}}\left(\frac{1-\e^{2Ks/3}}{1-\e^{2Kt/3}}\right)^{\frac{N}{2}}.
\]
\end{lemma}

\begin{proof}

Consider $f \in \Lonew$ such that $f \ge 0$ $\m_w$-a.e.~in $M$. Since $\m$ has a finite measure and $w$ is essentially bounded below away from zero, the measure $\m_w$ has a finite measure as well, so for any $\epsilon>0$, the function $f_\epsilon = f + \epsilon$ still belongs to $\Lonew$. For $z \in M$ and $\tau \ge 0$, set $$u_\epsilon (z,\tau) := \P^w_\tau f_\epsilon (z)$$ 
and notice that for any given $\tau >0$ it holds that $u_\epsilon(z,\tau)>0$ for $\m_w$-a.e.~$z \in M$.  By the BIP property of the original space $(M,\dist,\m)$, given $y,z \in M$ and $r>0$, there exists $\pi^r \in \mathrm{OptGeo}(\m_0^r,\m_1^r)$ concentrated on $\dist$-geodesics, where $$\m_0^r =\m(B_r(y))^{-1} 1_{B_r(y)}\m  \quad \text{and} \quad  \m_1^r=\m(B_r(z))^{-1} 1_{B_r(z)}\m.$$  Consider $s<t$ in $(0,+\infty)$. Acting as in \cite[(4.11)]{GarofaloMondino-13} with the metric measure space $(M,\dist_w,\m_w)$,
\begin{align*}
    \int \ln \left( \frac{u_\epsilon(\gamma_1,s)}{u_\epsilon(\gamma_0,t)}\right) \d \pi^r(\gamma) & \le \int \int_0^1 |d(\ln u_\epsilon(\cdot,t + \tau(s-t))) |_{w}(\gamma_\tau) |\dot{\gamma}|_{w}(\tau) \d \tau \d \pi^r(\gamma) \\
    & \phantom{=} - (t-s) \int \int_0^1 \left( \frac{\partial}{\partial \tau}  \ln u_\epsilon \right) (\gamma_\tau,t + \tau(s-t)) \d \tau \d \pi^r(\gamma).
\end{align*}
Apply the Li-Yau inequality from Proposition \ref{prop:LiYau} at time $\tau$ to the function $f_\epsilon$. Divide the resulting inequality by $e^{2K\tau/3}u_\epsilon^2$ to obtain  
\begin{equation*}
- \left(\frac{\partial}{\partial \tau} \ln u_\epsilon\right) (z,\tau)
\leq -e^{2K\tau/3} |d(\ln u_\epsilon(\cdot,\tau) )|_{w}^2(z) - \frac{NK}{3}\frac{e^{2K\tau/3}}{1-e^{2K\tau/3}} \, \cdot 
\end{equation*}
Together with Young's inequality with $\eta >0$, this gives
\begin{align*}
    \int \ln \left( \frac{u_\epsilon(\gamma_1,s)}{u_\epsilon(\gamma_0,t)}\right) \d \pi^r(\gamma) 
    & 
    \le \frac{\eta}{2}\int \int_0^1 |d(\ln u_\epsilon(\cdot,t + \tau(s-t)) )|_{w}^2(\gamma_\tau)  \d \tau \d \pi^r(\gamma) \\
    &\phantom{=}  + \frac{1}{2\eta} \int \int_0^1 |\dot{\gamma}|_{w}^2(\tau) \d \tau \d \pi^r(\gamma)\\
    & \phantom{=} - (t-s) \int \int_0^1 e^{-2K(t + \tau(s-t))/3}|d(\ln u_\epsilon(\cdot,t + \tau(s-t)) )|_{w}^2(\gamma_\tau) \d \tau \d \pi^r(\gamma) \\
    & \phantom{=}- (t-s) \int \int_0^1 \frac{NK}{3}\frac{e^{2K(t + \tau(s-t))/3}}{1-e^{2K(t + \tau(s-t))/3}}  \d \tau \d \pi^r(\gamma).
\end{align*}
Since
\[
- (t-s) e^{-2K(t + \tau(s-t))/3} \le - (t-s) e^{-2Kt/3}
\]
we may choose $\eta=2(t-s)\e^{2Kt/3}$ to cancel out the first and third terms in the previous right-hand side. We get
\begin{align*}
    \int \ln \left( \frac{u_\epsilon(\gamma_1,s)}{u_\epsilon(\gamma_0,t)}\right) \d \pi^r(\gamma) & \le 
    \frac{\e^{-2Kt/3}}{4(t-s)} \int \int_0^1 |\dot{\gamma}|_{w}^2(\tau) \d \tau \d \pi^r(\gamma)
    \\
    &\phantom{=}- (t-s) \int \int_0^1 \frac{N K}{3}\frac{e^{2K(t + \tau(s-t))/3}}{1-e^{2K(t + \tau(s-t))/3}}  \d \tau \d \pi^r(\gamma).\nonumber
\end{align*}
In order to conclude, note that
\begin{align*}
\lim\limits_{r \downarrow 0} \int \int_0^1 |\dot{\gamma}|_{w}^2(t) \d t  \d \pi^r (\gamma) & = \lim\limits_{r \downarrow 0} \int \int_0^1 e^{-2w(\gamma(t))}|\dot{\gamma}|^2(t) \d t  \d \pi^r (\gamma)\\
& \le e^{-2 \|w\|_\infty} \lim\limits_{r \downarrow 0} \int \int_0^1 |\dot{\gamma}|^2(t) \d t  \d \pi^r (\gamma) \\
& = e^{-2 \|w\|_\infty} \dist^2(y,z).
\end{align*}
Moreover, as pointed out in \cite{GarofaloMondino-13}, a direct computation shows that
\[
-(t-s) \int_0^1 \frac{N K}{3}\frac{e^{2K(t + \tau(s-t))/3}}{1-e^{2K(t + \tau(s-t))/3}}  \d \tau  = \frac{N}{2} \ln \left( \frac{1-e^{-2Kt/3}}{1-e^{-2Ks/3}} \right). 
\]
Therefore, letting $r \to 0$, we obtain
\[
\ln \left( \frac{u_\epsilon(y,s)}{u_\epsilon(z,t)}\right) \le \frac{e^{-2 \|w\|_\infty} \dist^2(y,z)}{4(t-s)e^{2Kt/3}} + \frac{N}{2} \ln \left( \frac{1-e^{2Kt/3}}{1-e^{2Ks/3}} \right).
\]
We get the conclusion by using that $e^{-2 \|w\|_\infty} \dist^2(y,z) \le \dist_w^2(y,z)$ and then by letting $\epsilon$ tend to $0$. 
\end{proof}

\begin{remark}
    Note that assuming $\k_T(M,\dist,\m,\kappa) = 0$ in Lemma~\ref{lem:LYtimechange} yields the form of the Li-Yau inequality that holds on smooth Riemannian manifolds with non-negative Ricci curvature.
\end{remark}

Next, we derive heat kernel bounds and a local volume doubling property. To this purpose, we use the exponential integrability of the heat semigroup of the time-changed space by adapting the arguments from Baudoin--Garofalo \cite[Theorem 3.1]{BaudoinGarofalo} to our setting. We generalize the latter result to spaces satisfying $\mathrm{BE}(-K,N)$ with $K>0$. The following arguments to obtain heat kernel bounds and volume doubling are standard, but since we are interested in the behavior of the constants, we give the details. 

We begin with the exponential integrability.

\begin{proposition}\label{lem:exp}
There exist $A=A(n)\in(0,1)$ and $C=C(n)>0$ such that for any $x\in \M$ and $r\in(0,1/\sqrt{KN}]$, for $\m_w$-a.e.~$y\in B_w(x,r/2)$,
    \[
    \P_{Ar^2}^w [\One_{B_w(x,r)}](y)\geq C.
    \]

\end{proposition}

\begin{proof}We employ Proposition~\ref{prop:gauge} with specific choices of $a$ and $\gamma$ and initial conditions for the heat equation. To this end, we split the proof into several steps.

\textbf{Step 1.} We prove that for all non-negative functions $f \in L^\infty(\M, \m_w)$, $\varphi \in \mathrm{Test}_\infty(M,\dist_w,\m_w)$, and all $\tau,t>0$,
\begin{align*}
    \tau\int\Phi^w(t)\varphi\d \m_w-(\tau+t)\int\Phi^w(0)\varphi\d \m_w
    & \geq-C_1\int [\Delta^w P_t^wf]\varphi\d\m_w  -\frac{N}{8}C_2\int [P_t^wf]\varphi\d \m_w
    \end{align*}
    where
    \[
    C_1 =  C_1(t,K,\tau):= t(1+2K(\tau+t)) \quad \text{and} \quad C_2 =  C_2(t,K,\tau):= (1+2K(\tau+t))^2\ln\left(1+\frac{t}{\tau}\right).
    \]

We apply Proposition \ref{prop:gauge} with $a : (0,t) \to \mathbb{R}_+$ and $\gamma : (0,t) \to \mathbb{R}$ defined by
    \[
    a(s)=\tau+t-s, \qquad \gamma(s)=-\frac{N}{4a(s)}(1+2Ka(s)).
    \]
    Observe that
    \[
    a'=-1, \qquad \frac{4a\gamma}{N}=-(1+2Ka),\qquad \frac{2a\gamma^2}{N}=\frac{N}{8a}(1+2Ka)^2.
    \]
    In particular,
    \[
    a' - \frac{4a\gamma}{N} - 2 K a = 0.
    \]
    In virtue of the dominated convergence theorem, we assume that $f \ge \delta$ $\m_w$-a.e.~for some $\delta>0$. Integrating the inequality from Proposition \ref{prop:gauge} from $0$ to $t$ with the previous $a$ and $\gamma$, we obtain
    
    \begin{align*}
    \tau\int\Phi^w(t)\varphi\d \m_w-(\tau+t)\int\Phi^w(0)\varphi\d \m_w & 
    \geq -\int_0^t(1+2Ka(s))\d s\int [\Delta^w P_t^wf]\varphi\d\m_w\\
    & \phantom{=} \, -\frac{N'}{8}\int_0^T\frac{(1+2Ka(s))^2}{a(s)}\d s\int [P_t^wf]\varphi\d \m_w.
    \end{align*}
Since
\[
\int_0^t(1+2Ka(s))\d s\leq t(1+2K(\tau+t)) = C_1 \]
and \begin{align*} \phantom{=} \int_0^t\frac{(1+2Ka(s))^2}{a(s)}\d s 
& \leq \left[ \max_{[0,t]}(1+2Ka)^2 \right]\int_0^t\frac{\d s }{a(s)}=(1+2K(\tau+t))^2\ln\left(1+\frac{t}{\tau}\right) = C_2,
\end{align*}
we get the desired result.

\textbf{Step 2.} For any $h\in L^\infty(\M,\m_w)$ such that $|dh|_{w}^2\le 1$ $\m_w$-a.e., any $\lambda,t >0$,
\[
h \le \psi(\lambda,t) + \Phi(\lambda,t)
\]
where
\[
\psi(\lambda,t)=\frac{1}{\lambda}\ln \P_t^w\e^{\lambda h}
\]
and
\[
\Phi(\lambda,t) = \int_0^t \lambda \eta (\lambda^2 t) + \frac{N}{8 \lambda} \ln \left( 1 + \frac{1}{\eta(\lambda^2 t)} \right) \left( \frac{1}{t} + 2 K (\eta(\lambda^2t )+1)\right) \d t, \quad \eta(x) = \frac{1}{2} \left( \sqrt{1 + \frac{1}{2x}}- 1\right).
\]

Notice that $\P_t^w \e^{\lambda h}=\e^{\lambda\psi(\lambda, t)}.$ Therefore, by Jensen's inequality, we have
\[
\P_t^wh=\frac{1}{\lambda}\P_t^w\ln \e^{\lambda h}
\leq \frac{1}{\lambda}\ln \P_t^w\e^{\lambda h}=\psi(\lambda,t) \quad \text{$\m_w$-a.e.~on $M$.}
\]
Apply Step 1 to the function 
$f = \e^{\lambda h}$ to obtain
\begin{multline*}
    \tau\lambda^2\int P_t^w(\e^{\lambda h}|dh|_w^2)\varphi\d \m_w
    -(\tau+t)\lambda^2\int \e^{\lambda \psi(\lambda,t)}|d\psi(\lambda,t)|_w^2\varphi\d \m_w \\
        \geq -C_1\int [\Delta^w P_t^w\e^{\lambda h}]\varphi\d\m_w -\frac{N}{8}C_2\int \e^{\lambda \psi(\lambda,t)}\varphi\d \m_w.
\end{multline*}
Note that the second term on the left-hand side is non-negative. We can therefore drop it to obtain a weaker inequality. Since $\P_t^w$ is positivity-preserving,
\[
\P_t^w(\e^{\lambda h}|d h|_w^2)\leq \P_t^w\e^{\lambda h}
=\e^{\lambda\psi(\lambda,t)}.\]
Morever,
\[
\Delta^w \P_t^w\e^{\lambda h} = \left. \frac{\partial }{\partial t} \right|_{s=t} \P_s^w\e^{\lambda h} = \lambda \e^{\lambda \psi(\lambda,t)}\frac{\partial\psi}{\partial t}(\lambda, t).
\]
Thus the above inequality yields
\[
    \tau\lambda^2\int \e^{\lambda\psi(\lambda,t)}\varphi\d \m_w
        \geq -C_1\int \lambda \e^{\lambda \psi(\lambda,t)}\frac{\partial\psi}{\partial t}(\lambda, t) \varphi\d\m_w -\frac{N}{8}C_2\int \e^{\lambda \psi(\lambda,t)}\varphi\d \m_w.
\]
Rearrange to get
\[
    \int \left( \frac{\partial\psi}{\partial t}(\lambda,t)  +\frac{\tau \lambda}{C_1} + \frac{NC_2}{8\lambda C_1} \right) C_1 \lambda\e^{\lambda \psi(\lambda,t)}\varphi\d \m_w\geq0 .
\]
Since this holds for any $\varphi\in \Linfw$ such that $\varphi \ge 0$ $\m_w$-a.e. we get that 
\[
   0  \leq  \frac{\partial\psi}{\partial t}(\lambda,t)  + \frac{\tau \lambda}{C_1} + \frac{NC_2}{8\lambda C_1} = \frac{\partial\psi}{\partial t}(\lambda,t)  + \frac{\tau \lambda}{t(1+2K(\tau+t))} + \frac{N}{8\lambda}\frac{(1+2K(\tau+t))}{t}\ln\left(1+\frac{t}{\tau}\right) \qquad \text{$\m_w$-a.e.}
\]
Choose
\[
\tau = t \eta (\lambda^2 t)
\]
to get
\[
    \frac{\partial\psi}{\partial t} (\lambda,t)
    +\frac{\lambda \eta (\lambda^2 t)}{1 + 2Kt(1+\eta(\lambda^2 t))} + \frac{N}{8 \lambda} \ln \left( 1 + \frac{1}{\eta(\lambda^2 t)} \right) \left( \frac{1}{t} + 2 K (\eta(\lambda^2t )+1)\right)
     \geq0.
   \]
   Since $1 + 2Kt(1+\eta(\lambda^2 t)) \ge 1$, we get
   \[
    \frac{\partial\psi}{\partial t} (\lambda,t)
    +G(\lambda,t)
     \geq0
   \]
   where
   \[
   G(\lambda,t) \df \lambda \eta (\lambda^2 t) + \frac{N}{8 \lambda} \ln \left( 1 + \frac{1}{\eta(\lambda^2 t)} \right) \left( \frac{1}{t} + 2 K (\eta(\lambda^2t)+1)\right).
   \]
 Integrating this inequality from $0\leq \epsilon\leq t$ yields 
\[\lambda\psi(\lambda,t)+\lambda\int_{\epsilon}^t G(\lambda,s)\d s\geq \lambda \psi(\lambda,\epsilon)\geq \lambda \P_{\epsilon}^wh\]
$\m_w$-almost everywhere. In particular, since $\P_\epsilon^wh$ converges in $L^2(\M,\mu^w)$ to $h$ as $\epsilon\to0$, there is a subsequence $(\epsilon_n)_n$ such that $\P_{\epsilon_n}^wh\to h$, $n\to\infty$, $\m_w$-almost everywhere. Hence
\[
h=
\lim_{n\to \infty}\P_{\epsilon_n}^w h
\leq \psi(\lambda,t)+\Phi(\lambda,t)
\]
almost everywhere, as claimed.

\textbf{Step 3.} For $\m_w$-a.e.~$y\in B_w(x,r/2)$, we have  \[
\P_t^w(\One_B)(y)\geq \e^{-\lambda r/2}(\e^{-\lambda \Phi(\lambda,t)(y)}-\e^{-\lambda r/2 }).
\]
Apply the previous step to $h=-\dist_w(x,\cdot)\wedge r$ where $x \in M$ and $r \in (0,1/\sqrt{K N}]$. Then $h\in L^\infty(\M,\m_w)$ with $|dh|_w \le 1$ $\m_w$-a.e.~and we clearly have
\[
\e^{\lambda h}\leq \e^{-\lambda r}\One_{B^c}+\One_B
\]
where $B=B_w(x,r)$.  Apply $\P_t^w$ for $t>0$ to get 
\[
\e^{\lambda \psi(\lambda, t)}=\P_t^w\e^{\lambda h}\leq \e^{-\lambda r }+\P_t^w(\One_B) \qquad \text{$\m_w$-a.e.}
\]
Since
\[
h\geq -r/2 \qquad \text{on}\  B_w(x,r/2),
\]
we have for $\m_w$-a.e.~$y\in B_w(x,r/2)$
\[
\e^{-\lambda r/2}\leq \e^{\lambda h(y)}\leq \e^{\lambda\psi(\lambda,t)(y)}\e^{\lambda \Phi(\lambda,t)(y)}
\]
and thus
\[
\P_t^w(\One_B)(y)\geq \e^{\lambda \psi(\lambda, t)(y)}-\e^{-\lambda r }
\geq \e^{-\lambda r/2-\lambda \Phi(\lambda,t)(y)}-\e^{-\lambda r }
=\e^{-\lambda r/2}(\e^{-\lambda \Phi(\lambda,t)(y)}-\e^{-\lambda r/2 }).
\]

\textbf{Step 4.} Finally, we conclude by finding the constants $A\in(0,1)$ and $C>0$. Choose $\lambda=1/r$ and $t=Ar^2$ to obtain for $\m_w$-a.e.~$y\in B_w(x,r/2)$
\[
\P_{Ar^2}^w(\One_B)(y)
\geq\e^{- 1/2}(\e^{-(1/r) \Phi(1/r,Ar^2)(y)}-\e^{- 1/2 }).
\]
In order to bound the right-hand side from below, we estimate $F(1/r,Ar^2)$ from above as follows. Write
\[
    \frac{1}{r} \Phi(1/r,Ar^2)=\frac{1}{r}\int_0^{Ar^2}G(1/r,\tau)\d \tau
    \\
    =S_1+S_2+S_3
    \]
    where
    \[S_1=\frac{1}{r^2}\int_0^{Ar^2}
    \eta\left( \frac{t}{r^2}\right)\d t
    ,\qquad S_2=\frac{N}{8}\int_0^{Ar^2}\frac{1}{t}\ln\left(1+\frac{1}{\eta\left( \frac{t}{r^2}\right)}\right) \d t,\]
    and 
    \[
     S_3=\frac{2KN}{8}\int_0^{Ar^2}\bigg( \eta\left( \frac{t}{r^2}\right) +1 \bigg) \ln\left(1+\frac{1}{\eta\left( \frac{t}{r^2}\right)}\right)  \d t.
\]
The substitution $s=t/r^2$ yields 
\[
S_1=\int_{0}^{A}
    \eta(s) \d s
    ,\qquad S_2=\frac{N}{8}\int_0^A \ln \left(1 + \frac{1}{\eta(s)} \right) \frac{\d s}{s},
\]
and 
\[
     S_3=\frac{2KN}{8}r^2\int_0^{A}\bigg( \eta\left( s\right) +1 \bigg) \ln\left(1+\frac{1}{\eta\left(s\right)}\right)  \d s
\]
Since $\eta(s)\sim 1/\sqrt s$ as $s\to 0$, the integrals appearing in $S_1$, $S_2$, and $S_3$ are finite for $A>0$ and independent of $N$ and $K$. Recall that $N=n+3(1\vee (n-2)^2) \k\leq n+3\frac{(1\vee (n-2)^2)}{1\wedge (n-2)} =:\omega(n) $. Hence, for any $\epsilon>0$, there is a constant $A\in(0,1)$ such that 
\[
S_1\leq \epsilon, \qquad S_2\leq \epsilon \omega(n), \qquad \text{and}\qquad S_3\leq \epsilon KN r^2.
\]
Hence, for $\m_w$-a.e.~$y\in B_w(x,r/2)$,
\[
\P_{Ar^2}^w(\One_B)(y)\geq \e^{\lambda \psi(\lambda, t)(y)}-\e^{-\lambda r }
\geq \e^{-1/2}(\e^{-(1+\omega(n))\epsilon -\epsilon KNr^2)}-\e^{-1/2 }).
\]
This yields the claim.
\end{proof}

Next, we obtain heat kernel bounds, volume doubling and local $L^2$ Poincaré inequality via the Li-Yau estimate and the exponential integrability. From now on, we use the shorthand 
\begin{align*}
V_w(x,r) \df \m_w(B_w(x,r)).    
\end{align*}

\begin{lemma}
    There exists a constant $c_n>0$ such that, for all $t\leq 4/(KN)$, $\m_w$-a.e.~$x\in M$, and $\m_w$-a.e.~$y,z\in B_w(x,\sqrt t/2)$,
    \[
    \frac{c_n^{-N}}{V_w(x,\sqrt{t})}
    \leq p_t^w(y,z)\leq \frac{c_n^{N}}{V_w(x,\sqrt{t})}.
    \]
    Moreover, for $\m_w$-a.e.~$x\in \M$ and $r\leq 1/(KN)$, we have 
    \[
    V_w(x,2r)\leq c_n^{N}V_w(x,r).
    \]
    Finally, we have the local Poincar\'e inequality: for $\m_w$-a.e.~$x\in \M$ and $r\leq 1/(KN)$, 
    \[
    \|\varphi-\bar\varphi\|_{L^2(B_w(x,r),\m_w)}^2\leq c_N\int_{B_w(x,2r)}\Gamma^w(\varphi)\d\m_w
    \]
    for all $\varphi\in\Test(\M)$, where $\bar\varphi=\m_w(B_w(x,r))^{-1}\int_{B_w(x,r)}\varphi\d\m_w$.
\end{lemma}
\begin{proof}
    Let $x\in\M$, $s,t>0$. Then, for $\m_w$-a.e.~$y\in M$, we have
    \begin{align*}
    \P_{s}^w\One_{B_w(x,\sqrt t/2)}(y)& =\int p_{s}^w(y,z)\One_{B_w(x,\sqrt t/2)}(z)\d \m_w(z)
    \\
    & \leq 
    \left(\int p_{s}^w(y,z)^2\d \m_w(z)\right)^{\frac{1}{2}}V_w(x,\sqrt{t})^\frac{1}{2}\\
    & =p_{2s}^w(y,y)^\frac{1}{2}V_w(x,\sqrt{t})^\frac{1}{2}.
    \end{align*}
    By Proposition~\ref{lem:exp}, there are $A=A(n)\in(0,1)$ and $C=C(n)>0$ such that, for all $t\leq 4/(KN)$ and $\m_w$-a.e.~$y\in B_w(x,\sqrt t/2)$,
    \[
    C\leq \P^w_{At/4}\One_{B_w(x,\sqrt t/2)}(y)\leq p^w_{At/4}(y,y)^\frac{1}{2}V_w(x,\sqrt{t})^\frac{1}{2}.
    \]
    Apply the Harnack inequality for $p^w$ from Lemma~\ref{lem:LYtimechange} to obtain
    \[
    \frac{c^2}{V_w(x,\sqrt{t})}\le p^w_{At/4}(y,y)
    \leq p_t^w(y,z)
    \e^{\frac{\dist_w(z,y)^2}{4(t-At/4)\e^{2Kt/3}}}\left(\frac{\e^{2Kt/3}-1}{\e^{KAt/6}-1}\right)^{\frac{N}{2}}
    \]
    for $\m_w$-a.e.~$z,y\in B_w(x,\sqrt t/2)$ and all $t\leq 4/(KN)$. Estimating the factors on the right-hand side further, we obtain
    \[
    \e^{\frac{\dist_w(z,y)^2}{4(t-At/4)\e^{2Kt/3}}}\leq \e^{\frac{1}{(4-A)\e^{2Kt/3}}}\leq 1
    \]
    and, by the mean-value theorem,
    \[
    \frac{\e^{2Kt/3}-1}{\e^{KAt/6}-1}
    =\frac{2}{A}\frac{\e^{2Kt/3}-1}{2Kt/3-0}\frac{KAt/6-0}{\e^{KAt/6}-1}
    \leq 
    \frac{2}{A}\e^{2Kt/3}\leq \frac{2}{A}\e^{2K4/(3Kn)}=\frac{2}{A}\e^{8/3n},
    \]
    where we used $t\leq 4/(KN)\leq 4/(Kn)$. Hence, the lower bound follows.

    Next, we consider the heat kernel upper bound. By Lemma~\ref{lem:LYtimechange}, for $\m_w$-a.e.~$y,z,v\in B_w(x,\sqrt t)$ and $0<t\le T/4$, we have 
\[
p_t^w(y,z)\leq \e^{\frac{\dist_w(z,y)^2}{4t\e^{2K{2t}/3}}}\left(\frac{\e^{4Kt/3}-1}{\e^{2Kt/3}-1}\right)^{\frac{N}{2}}p_{2t}^w(z,z)\leq \e^{4N/3n}p_{2t}^w(z,z)\leq \e^{8N/3n}p_{4t}^w(z,v)
\]
    Integration in $v \in B_w(x,\sqrt t)$ with respect to $\m_w$ yields for $\m_w$-a.e.~$y,z\in B_w(x,\sqrt t)$
    \[p_t^w(y,z)V_w(x,\sqrt{t})\leq c_n^{N}\int_{B_w(x,\sqrt t)}p^w_{4t}(z,v)\d\m_w(v)\leq c_n^{N},
    \]
    i.e., the upper bound. Finally, we are concerned with the doubling property. Setting $t=4r^2\leq 4/(KN)$, we have, for $\m_w$-a.e.~$z\in B_w(x,r)$,
    \[
    V_w(x,2r)\leq \frac{c_n^{N}}{p_{4r^2}(z,z)}\leq \frac{c_n^{2N}}{p_{r^2}(z,z)}\le c_n^{3N}V_w(x,r).
    \]
    The statement about the Poincar\'e inequality is by now standard, and the constant is independent of $N$ since $N\le c_n$ by its definition and the assumption on $\k$. This implies the claim.
\end{proof}

The spaces $(\M,\dist)$ and $(\M,\dist_w)$ are quasi isometric. Moreover, we have $\Ch=\Ch^w$.  This yields the volume doubling property and weak Poincar\'e inequalities on the original space $(\M,\dist,\m)$. Furthermore, \cite{Sturm95_II,Sturm-96} yields localized full Gaussian heat kernel bounds and the scale-invariant Harnack inequality on $(\M,\dist,\m)$.

\begin{theorem}\label{thm:main1}
    Let $(\M,\dist,\m)$ be a locally compact length infinitesimally Hilbertian space with strongly regular Dirichlet form $\Ch$ and intrinsic metric~$\dist$. Assume  $\m(\M)<\infty$ and  $(\M,\dist,\m)$ satisfies $\BE(-\kappa,n)$ for some $\kappa \in \mathcal{M}_0(M,\dist,\m)$ such that
\[\k_T= \k_T(M,\dist,\m,\kappa)<1\wedge \frac{1}{2|c-2|}.\]
  \begin{enumerate}
        \item
    The space $(\M,\dist,\m)$ satisfies the local doubling property and a family of local weak Poincar\'e inequalities. More precisely, there exist constants $c_1=c_1(\k_T,n,\mu)\geq 1$, $c_2=c_2(\k_T,n)\geq 1$ and $N=N(\k_T,n)\geq n$ such that, for all $0<r\leq R\leq (4/(KN))\wedge (\diam\M/c_2)$, we have
    \[
    \frac{V(x,R)}{V(x,r)}\leq c_1\left(\frac{R}{r}\right)^{N},\qquad \int_{B(x,r)}|\varphi-\bar\varphi|^2\d\m\leq c_1r^2\int_{B(x,c_2 r)}\Gamma(\varphi)\d\m.
    \]
    In particular, $N\to n$ when $\k_T\to 0$.
     \item The space $(\M,\dist,\m)$ satisfies a scale-invariant parabolic Harnack inequality. More precisely, there exist constants $c_1=c_1(\k_T,N,\mu)>0$ and $c_2=c_2(\k_T,N,\mu)>0$ such that, for all $x\in M$, $t\in\RR$, $r\in(0,(4/(KN))\wedge (\diam\M/c_2)]$, and all $u$ solving $\partial_tu=\Delta u $ weakly in $B(x,2r)$, we have 
    \[
    \sup_{[t-r^2,t-3r^2/4]\times B(x,r)} u \leq c_1 \inf_{[t-r^2/4,t]\times B(x,r)} u.
    \]
    \item The heat kernel $p$ of $(\M,\dist,\m)$ satisfies double-sided Gaussian heat kernel bounds. More precisely, there exist constants $c_1=c_1(\k_T,n)\ge 1$ and $c_2=c_2(\k_T,n)\geq 1$ such that, for $\m$-a.e.~$x,y\in \M$ and $0<t\leq (4/(KN))\wedge (\diam\M/c_2)$:
    \[
    \frac{c_1^{-1}}{V(x,\sqrt{t})}\e^{-c_1\frac{d(x,y)^2}{t}
    }
    \leq p_t(x,y)
    \leq 
    \frac{c_1}{V(x,\sqrt{t})}\e^{-\frac{d(x,y)^2}{c_1t}
    }.
    \]
    \end{enumerate}
\end{theorem}

\begin{proof}
 1. The previous lemmma yields for all $x\in\M$ and $0<r\leq R<(4/(KN))\wedge (\diam\M/\e^{2\|w\|_\infty})$
    \[
            \frac{V(x,R)}{V(x,r)}
           \leq \e^{2\|w\|_\infty}\frac{V_w(x,\e^{2\|w\|_\infty}R)}{V_w(x,r)}
           \leq \e^{2\|w\|_\infty}\e^{2\|w\|_\infty N}\left(\frac{R}{r}\right)^{N},
        \]
        where we used $w\geq 0$ a.e.~in the first inequality, and the quasi isometry properties of $\dist$ and $\dist_w$ in the second inequality. Thus, $(\M,\dist,\m)$ has the local doubling property. 
        \\
        Next, we aim for the weak Poincar\'e inequality. Let $r<(4/(KN))\wedge (\diam\M/\e^{2\|w\|_\infty})$
        and denote $r_w:=\e^{2\|w\|_\infty}r$. By the definition of $\m_w=\e^{2w}\m$ and the quasi isometry properties of $\dist$ and $\dist_w$, for all $x\in\M$ and $\phi\in\Test(\M,\dist,\m)$, we have
        \begin{align*}
            \int_{B(x,r)}|\varphi -\bar\varphi|^2\d\m
            & =
            \inf_{a\in\RR}\int_{B(x,r)}|\varphi -a|^2\d\m
           \\
           & \leq \inf_{a\in\RR}\int_{B_w(x,r_w)}|\varphi -a|^2\d\m_w\\
           & 
=\int_{B_w(x,r_w)}|\varphi -\bar\varphi|^2\d\m_w
\\
& \leq c_n r_w^2\Ch^w(\varphi)\\
& =c_n\e^{4\|w\|_\infty}r^2\Ch(\varphi).
        \end{align*}
    
2.  This is a direct consequence of 1.~and \cite[Theorem~3.5]{Sturm-96}.

3. The heat kernel lower bound follows from \cite[Theorem~4.8]{Sturm-96}. Since $\dist$ is intrinsic and metric balls are precompact, volume doubling and weak Poincar\'e imply the scale-invariant Sobolev inequalities. Hence, the heat kernel upper bound follows from \cite[Theorem~2.4]{Sturm95_II}.
\end{proof}

\begin{remark} We do not know whether the Li-Yau inequality on $(\M,\dist,\m)$ can be obtained without constructing a time-changed space. In other contexts, this inequality was obtained directly, see e.g.~the works \cite{BaudoinGarofalo,GarofaloMondino-13, BBG17,Rose19,Carron-19}.
\end{remark}

\subsection*{AI disclosure}

No AI was used during the preparation of this article.

\bibliographystyle{alpha}
\bibliography{biblio}

\end{document}